\documentclass{amsart}
\usepackage[utf8]{inputenc}
\usepackage{cite}
\usepackage{amssymb,amsmath,amsthm}

\newcommand{\amsprimary}[1]{{\footnotesize\noindent AMS 2010 \textit{Mathematics subject
classification:} Primary #1\vspace{1pc}}}
\newcommand{\keywordsnames}[1]{{\footnotesize\noindent\textit{Key words:} #1\vspace{1pc}}}

\DeclareMathOperator{\sgn}{sgn}

\newtheorem{theorem}{Theorem}
\newtheorem{teo}{Theorem}
\newtheorem{prop}[teo]{Proposition}
\newtheorem{cor}[teo]{Corollary}
\newtheorem{lema}[teo]{Lemma}

\theoremstyle{definition}

\theoremstyle{remark}

\title{Stability of geometric flows on the circle}
\author{Jean Cortissoz \and C\'esar Reyes}
\thanks{The second author wants to thank the Fondo de Investigaciones de la Facultad de  Ciencias  de  la  Universidad  de  los  Andes for funding this project by means of: "Convocatoria  2017-2  para  la  Financiación  de proyectos de  Investigación  Categoría  Estudiantes de Doctorado", project named:"Flujo de Ricci sobre el cilindro con frontera".}
\date{}
\date{}

\begin{document}

\maketitle

\begin{abstract}
    In this paper we prove a general stability result for higher order geometric flows on the
    circle, which basically states that if the initial condition
    is close to a round circle, the curve evolves smoothly and exponentially fast 
    towards a circle, and we improve on known
    convergence rates (which we believe are almost sharp). The polyharmonic flow is an instance of the flows to which our result can be applied.
\end{abstract}
 
{\keywordsnames {geometric flows in the circle; stability; convergence rate.}}

{\amsprimary {53C44; 35K55; 58J35.}}

\section{Introduction}

The study of deformation of curves by using flow methods has become an important area of research in Geometric Analysis. In general the
purpose is to deform the curve as follows.
We let 
\[
x:\,\mathbb{S}^1\times\left[0,T\right)\longrightarrow \mathbb{R}^2
\]
be a family of smooth convex embeddings of $\mathbb{S}^1$, the unit circle, into $\mathbb{R}^2$.
and we assume that this family of embeddings satisfy 
an equation of the form
\begin{equation}
\label{generalflow}
\frac{\partial x}{\partial t}=F\left(k,k_{s},\dots, k^{\left(n\right)}_s\right)N,
\end{equation}
where $k$ is the curvature of the embedding and $N$ is the normal vector pointing outwards the region bounded by $x\left(\cdot,t\right)$ and $k^{\left(m\right)}_s$ denotes the $m$-th derivative of $k$ with respect to the arclength parameter.

The dean of all these family of flows is the curve shortening flow, which occurs when in (\ref{generalflow})
we make
$F=-k$. The curve shortening flow has been widely studied and is very well understood 
(\cite{Hamilton}). 
The family of $p$-curve shortening flows is obtained by making
$F=-\frac{1}{p}k^p$ and the family of polyharmonic flows (which gives an example 
of an application of our main result) is obtained with
$F=\left(-1\right)^{p+1}k^{\left(2p\right)}_s$.

In most cases, as in the case of the 
curve shortening and $p$-curve shortening flow 
(see \cite{Andrews, Hamilton})
this flows do what is expected of them: they deform convex embedded curves into circles
(after normalisation).
However, some interesting behavior occurs in the higher order flow (see 
for instance \cite{Giga, Wheeler2}), and even proving that solutions starting close
to a circle exist globally and converge to a circle can be challenging (see \cite{Elliott}).

Our purpose in this paper is to look at the structure of the equation for the curvature of a flow
of the form (\ref{generalflow}),
so that the following stability property can be deduced:
If we start close to a circle (in a sense to be specified below), the curve evolves 
smoothly and exponentially fast towards a circle (though perhaps not the same circle the curve started close to). 

With this considerations in mind, 
the semilinear parabolic equations we shall consider are of the following general form,
\begin{equation}
\label{eq1:curvatureflow}
\left\{
\begin{array}{l}
\dfrac{\partial k}{\partial t}=\left(-1\right)^{p+1}k^{M}\dfrac{\partial^{2p}k}{\partial \theta^{2p}}
+G\left(k,\dfrac{\partial k}{\partial \theta},\dots,\dfrac{\partial^{2p-1}k}{\partial \theta^{2p-1}}\right)
\quad \left[0,2\pi\right]\times\left(0,T\right),\\
k\left(\theta,0\right)=\psi\left(\theta\right),
\end{array}
\right.
\end{equation}
endowed 
with periodic boundary conditions. In our case $G\left(z_0,\dots,z_{2p-1}\right)$ is any polynomial in $2p$ variables for which the following conditions
hold. Any term $z_0^{\alpha_0}z_1^{\alpha_1}\dots z_{2p-1}^{\alpha_{2p-1}}$ of $G$ satisfies
\begin{itemize}\label{condition}
\item
either $\alpha_1+\alpha_2+\dots+\alpha_{2p-1}> 1$,
\item
or in the case that $\alpha_1+\alpha_2+\dots+\alpha_{2p-1}=1$, there is a $j$ such that
$\alpha_{2j}=1$.
\end{itemize}

 We shall also require something else from solutions to (\ref{eq1:curvatureflow}),
 but before revealing what our requirement is, and to put our results into context, let us first analyze
 the linearisation of the operator on the right hand side of (\ref{eq1:curvatureflow}) around the solution $k\equiv 1$, which is given by
\[
\mathcal{L}u=\left(-1\right)^{p+1}\frac{\partial^{2p}u}{\partial \theta^{2p}}+
\sum_{j=0}^{2p-1}\frac{\partial G}{\partial z_j}\left(1,0,\dots,0\right)\frac{\partial^j u}{\partial \theta^j}.
\]
Hence, the eigenvalues of $\mathcal{L}$ of the eigenfunctions $e^{in\theta}$ with $n\in\mathbb{Z}$
can be computed as
\[
\lambda_n
=-n^{2p}+\sum_{j=0}^{p-1}\left(-1\right)^j\frac{\partial G}{\partial z_{2j}}\left(1,0,\dots,0\right)n^{2j}.
\]
In the case of the higher order curvature flows we want to study, besides having $\lambda_0=0$ we also have that 
$\lambda_1=0$. It is because of this that it might not occur, in principle, that solutions with initial
data close to the constant function 1 will evolve towards a constant function.

On the other hand, solutions 
to (\ref{eq1:curvatureflow}), as long as $k$ is 
the curvature function of a simple closed curve and remains positive, satisfy the identities
\begin{equation}
\label{geometriccondition}
U_{\pm}\left(k\left(\theta,t\right)\right):=\int_{0}^{2\pi}\frac{e^{\pm i\theta}}{k\left(\theta,t\right)}\,d\theta=0.
\end{equation}
These identities are the key to showing exponential convergence of the curvature function
$k$ towards a constant function if 
the initial data $\psi$ is close to a constant function. They give us control over the Fourier coefficient of
the wavenumbers $n=\pm 1$.
So we will require that solutions to (\ref{eq1:curvatureflow}) also satisfies (\ref{geometriccondition}).
We shall call these flows {\it geometric flows}. 

Henceforth, we will use the following semi-norms 
\[
\left\|f\right\|_{\beta}=\max\left\{\sup_{n\neq 0}\left|n\right|^{\beta}\left|Re\left\{\hat{f}(n)\right\}\right|, \sup_{n\neq 0}\left|n\right|^{\beta}\left|Im\left\{\hat{f}(n)\right\}\right|\right\},
\]
where $\hat{f}\left(n\right)$ represents the $n$-th Fourier coefficient of $f$.

Notice also that
\[
\frac{\partial G}{\partial z_{2l}}\left(1,0,\dots,0\right)=\sum_{j_l} a_{j_l}
\]
where (and this notation will be important in the statement of our main result)
\[
a_{j_{l}}\quad \mbox{is the coefficient of the term} \quad k^{j_{l}}\dfrac{\partial^{2l}k}{\partial \theta^{2l}}.
\]
We can now state our main theorem. 
\begin{theorem}
\label{maintheorem}
Consider equation (\ref{eq1:curvatureflow}) with initial data $\psi$ ($U_{\pm}\left(\psi\right)=0$)
a geometric flow.
There exists a $\delta>0$, which may also depend on the $L^{\infty}$-norm of $\psi$, such that if 
\[
\left\|\psi\right\|_{2p+1}\leq \delta\hat{\psi}\left(0\right),
\]
and there exists a $c<0$ such that for all integers $n\geq 2$
\[
\begin{aligned}
P_{n,4\delta}= &-n^{2p}\left(1-4\delta\right)^M\hat{\psi}^M(0)+&\\ 
&\sum_{l=1}^{p-1}\sum_{j_l}\left(-1\right)^l n^{2l}a_{j_l}\left(1+
\sgn\left(\left(-1\right)^l a_{j_l}\right)4\delta\right)^{j_l}\hat{\psi}^{j_l}\left(0\right)<c,&\end{aligned}
\]
then the solution of (\ref{eq1:curvatureflow}) with initial data $\psi$ exists for all time and converges exponentially to a constant
that we shall denote by
$\hat{k}\left(0,\infty\right)$. Moreover, if for every $0<\epsilon<4\delta$
we have the inequality
\[
P_{2,\epsilon}\geq P_{n,\epsilon}
\]
then 
we have the following estimate on the rate of convergence of the solution towards a constant:
for any $\epsilon>0$ there exists a $C_{l,\epsilon}$ such that
\[
\left\|k\left(\theta,t\right)-\hat{k}\left(0,\infty\right)\right\|_{C^l\left(\mathbb{S}^1\right)}\leq C_{l,\epsilon} e^{P_{2,\epsilon} t},
\quad
l= 0, 1, 2, \ldots.
\]
\end{theorem}
An aside is in place here: to verify that the assumption on $P_{n,4\delta}$ holds for $\delta>0$ small, it
is enough to verify that it holds for $\delta=0$.

As we said before, the main difficulty to prove Theorem \ref{maintheorem} is the fact that the eigenvalue $\lambda_1$ is equal to 0, and 
thus we cannot expect smooth exponential convergence of the function $k\left(\cdot,t\right)$ towards a constant. 
The reader should compare our discussion with 
Theorem 1 (and the remark right after its statement) in Simon's celebrated paper
\cite{Simon}.
We deal with this issue using the approach presented in \cite{cortissoz17} for the $p$-curve shortening flow: that
is, we make use of (\ref{geometriccondition}) to control the Fourier coefficients corresponding to the
wavenumber $n=\pm 1$ in terms of the Fourier coefficients for the higher wavenumbers; one difference
with the family of flows considered in \cite{cortissoz17} to those considered in this paper, is
the fact that in that family of flows the average of the curvature blows-up and this helps on
sharpening the convergence rates, whereas in the flows considered in this
paper the average of the curvature remains bounded (but also large enough). 
In any case, though we do not obtain sharp rates of convergence, our method allows
us to give better convergence rates than the known ones (which are basically close to the
expected rate $e^{\lambda_2t}$).

Notice that the requirement in Theorem \ref{maintheorem} on how close the initial condition must be to a constant (in this case its average)
can be also recasted in terms of the 
H\"older 
$C^{2p,\alpha}\left(\mathbb{S}^1\right)$ spaces ($0<\alpha<1$) or the Sobolev spaces $H^{2p+1}\left(\mathbb{S}^1\right)$. 

To show how our main theorem applies, we will first consider the 
case of the polyharmonic flows. For polyharmonic flows, assuming that
the initial curve encloses an area of $\pi$, then $\hat{k}\left(0,\infty\right)=1$ (see Theorem 1 in \cite{Wheeler}), and 
\[
P_{\epsilon}=-2^{2p+2}\left(1-\epsilon\right)^{2p+2}+\left(1+\epsilon\right)^{2p+2}2^{2p},
\]
so we obtain the following corollary.
\begin{cor}
\label{applicationpolyharmonic}
Consider equation (\ref{eq1:curvatureflow}) corresponding to the 
case when $$F=\left(-1\right)^{p} k_{s}^{\left(2p\right)},$$ i.e. when (\ref{eq1:curvatureflow})
represents the evolution of the curvature in the case of the polyharmonic flows, with initial
condition $\psi$. There exists a $\delta>0$, which may depend on the
$L^{\infty}$ norm of $\psi$, such that if the curvature of the initial curve satisfies
\begin{equation}
\label{initialconditionregularity}
\left\|\psi\right\|_{2p+3}\leq \delta \hat{\psi}\left(0\right),
\end{equation}
and the area enclosed by it is equal to $\pi$,
then the solution of (\ref{eq1:curvatureflow}) with initial data $\psi$ exists for all time and converges smoothly and exponentially 
towards the constant function 1 (i.e., the curve is evolving towards a unit circle). 
Moreover, we have the following estimate on the rate of convergence of the solution towards a constant:
for any $\epsilon>0$ there exists a $C_{l,\epsilon}$ such that
\[
\left\|k\left(\theta,t\right)-1\right\|_{C^l\left(\mathbb{S}^1\right)}\leq C_{l,\epsilon} e^{P_{\epsilon} t},
\quad
l= 0, 1, 2, \ldots;
\]
where $P_{\epsilon}$ is given by
\[
\begin{aligned}
 -2^{2p+2}\left(1-\epsilon\right)^{2p+2}+\left(1+\epsilon\right)^{2p+2}2^{2p}.\end{aligned}
\]
\end{cor}
We must point out that the rates given in Corollary \ref{applicationpolyharmonic} improve on those obtained in \cite{Wheeler, Wheeler2}. 
Also, condition (\ref{initialconditionregularity}) is not sharp (we will discuss this in the last section of this paper).

On the other hand, Theorem \ref{maintheorem} allows us to show stability results for families of flows by
only having to check the form of the function $F$ in (\ref{generalflow}). This will be shown in the
last section of this work.

This paper is organised as follows. In Section 2 we give a proof of our main result; in Section 3 we apply it 
to the polyharmonic flow, give an example of another flow to which our main result applies, and propose families
of flows for which stability (in the sense described above) holds (see Propositions \ref{family1}
and \ref{family2} at the end of the paper); in Section 4 we will give a brief discussion on how
condition (\ref{initialconditionregularity}) might be improved in the case of the polyharmonic flow and 
a flow recently studied in \cite{AndrewsWheeler}.

\section{Proof of Theorem \ref{maintheorem}}

\subsection{Preliminaries}
Applying the Fourier transform to the evolution equation (\ref{eq1:curvatureflow}), we obtain
\[
\frac{\partial\hat{k}}{\partial t} (\xi,t)= (-1)^{p+1}\hat{k}^{*M}(\xi,t)*\left(i\xi\right)^{2p}\hat{k}(\xi,t)+ \hat{G}(\xi,t).
\]
If $G$ has the form
\[
G\left(k, \frac{\partial k}{\partial \theta}, \ldots, \frac{\partial^{2p-1} k}{\partial \theta^{2p-1}}\right) = \sum_{j_l\in\{0,1, \ldots, m\}}a_{j_0j_1 \dots j_{2p-1}}k^{j_0}\left[\frac{\partial k}{\partial \theta}\right]^{j_1} \cdots \left[\frac{\partial^{2p-1} k}{\partial \theta^{2p-1}}\right]^{j_{2p-1}},
\]
then $\hat{G}$ is given by
\[
\hat{G}(\xi,t) = \sum_{j_l\in\{0,1,\ldots,m\}} a_{j_0j_1 \dots j_{2p-1}}\hat{k}^{*j_0}*\left[i\xi\hat{k}\right]^{*j_1} * \cdots * \left[\left[i\xi\right]^{2p-1}\hat{k}\right]^{*j_{2p-1}}.
\]
The $l$-term in parenthesis in the last sum can be computed as
\[
\left[\left[i\xi\right]^l\hat{k}\right]^{*j_l} = i^{lj_l}\sum_{\sum q_{s,t}=\xi} \left[q_{l,1}q_{l,2} \cdots q_{l,j_l}\right]^l\hat{k}(q_{l,1},t)\hat{k}(q_{l,2},t) \cdots \hat{k}(q_{l,j_l},t).
\]
Then the evolution of $\hat{k}\left(\xi,t\right)$ equation turns out to be
\small
\[
\begin{aligned}
&\partial_t \hat{k} (\xi,t) \\&= -\sum_{q \in \mathbb{Z}} q^{2p}\hat{k}(q,t)\hat{k}^{*M}(\xi - q,t) \\ &\quad + \sum_{q_{l,h} \in \mathbb{Z}} a_{j_0j_1 \dots j_{2p-1}}i^{j_1 + \cdots + (2p-1)j_{2p-1}} \left[q_{1,1}q_{1,2} \cdots q_{1,j_1}\right] \cdots \left[q_{2p-1,1}q_{2p-1,2} \cdots q_{2p-1,j_{2p-1}}\right]^{2p-1} \\ & \quad \quad \left[\hat{k}(q_{0,1},t)\hat{k}(q_{0,2},t) \cdots \hat{k}(q_{0,j_0},t)\right] \cdots \left[\hat{k}(q_{2p-1,1},t)\hat{k}(q_{2p-1,2},t) \cdots \hat{k}(q_{2p-1,j_{2p-1}},t)\right]\\
&= -\sum_{q \in \mathbb{Z}} q^{2p}\hat{k}(q,t)\hat{k}^{*M}(\xi - q,t) + \sum a_J i^{J\cdot(0,1,...,2p-1)} Q_1 \cdots Q_{2p-1}^{2p-1} \hat{k}(Q_0,t) \cdots \hat{k}(Q_{2p-1},t)\\
&= -\sum_{q \in \mathbb{Z}} q^{2p}\hat{k}(q,t)\hat{k}^{*M}(\xi - q,t) + \sum a_J i^{J \cdot R} Q^R \hat{k}(Q,t),
\end{aligned}
\]
\normalsize
here $J=\left(j_0,j_1,\ldots,j_{2p-1}\right)$, $R=(0,1, \ldots, 2p-1)$, $Q=\left(Q_0,Q_1,\ldots,Q_{2p-1}\right)$ and $Q_l = \left(q_{l,0},q_{l,1},\ldots,q_{l,j_l}\right)$, and we adopt the notation
\[
\hat{k}(Q,t)=\hat{k}(Q_0,t) \cdots \hat{k}(Q_{2p-1},t).
\]
\\
Let $\mathcal{Z}$ be a finite subset of integers that contains 0, and such that if $n \in \mathcal{Z}$ then $-n \in \mathcal{Z}$. Also, we will consider the following sets:
\[
B_n=\left\{Q = \left(q_{1,1},\ldots,q_{1,j_1}, \ldots, q_{2p-1,1},\ldots,q_{2p-1,j_{2p-1}}\right):\, \sum_{l,k} q_{l,k} = n\right\}
\]
\[
A_n =\bigg\{Q \in B_n:\,\text{Q has at least two nonzero components}\bigg\}
\]
Using these sets we consider the following finite dimensional approximation to the equation considered above for
$\hat{k}\left(n,t\right)$
\[
\begin{aligned}
\partial_t \hat{k} (n,t) &= -n^{2p}\hat{k}^{M}(0,t)\hat{k}\left(n,t\right) + 
\sum_{l=1}^{p-1}\sum_{j_l}\left(-1\right)^ln^{2l}a_{j_l}\hat{k}^{j_l}\left(0,t\right)\hat{k}\left(n,t\right)
\\& \quad - \sum_{q \in \mathcal{Z}-\{n\}} q^{2p}\hat{k}(q,t)\hat{k}^{*M}(n - q,t) + \sum_{A_n \cap \mathcal{Z}^{2p-1}} a_J i^{J \cdot R} Q^R \hat{k}(Q,t),
\end{aligned}
\]
for $n\neq 0$, and 
\[
\partial_t\hat{k}(0,t) = - \sum_{q \in \mathcal{Z}-\{n\}} q^{2p}\hat{k}(q,t)\hat{k}^{*M}(-q,t) + \sum_{A_0 \cap \mathcal{Z}^{2p-1}} a_J i^{J \cdot R} Q^R \hat{k}(Q,t),
\]
for $n=0$. The proofs given below are for these finite dimensional approximations. However, our estimates
are strong enough to pass to the limit (see \cite{cortissoz13}).

\subsection{Important Remark} We shall be using the following fact freely.
The sums (here $l\geq 1$) satisfy
\[
\sum^* \frac{1}{q_1^2}\frac{1}{q_2^2}\dots\frac{1}{q_l^2}\frac{1}{\left(n-q_l-q_{l-1}-\dots q_1\right)^2}
=O\left(\frac{1}{n^2}\right),
\]
and 
\[
\sum^* \frac{1}{q_1}\frac{1}{q_2^{2}}\dots\frac{1}{q_l^2}\frac{1}{\left(n-q_l-q_{l-1}-\dots q_1\right)^2}
=O\left(\frac{1}{n}\right),
\]
where $*$ means that the sum is over all the $q_1,q_2,\dots, q_k$ such that
$\left|q_j\right|\geq 1$ and $\left|n-q_k-\dots-q_1\right|\geq 1$ (this
can be proved by induction, see Lemma 3.1 in \cite{cortissoz13}).

\subsection{The Proof}
We begin by a series of lemmas whose purpose is to show how the Fourier wavenumbers
decay exponentially in time. The first two lemmas show that the set of wavenumbers $n=\pm 1$ and $\left|n\right|\geq 2$
control each other: if the size of the Fourier coefficients with wavenumber in one set is small, so are
the Fourier coefficients with wavenumbers in the other set. Once this is proved, the Trapping Lemma 
(Lemma \ref{trappinglemma}) guarantees that
if we start close enough to $\hat{\psi}\left(0\right)$ (the
average of the initial data), the Fourier coefficients of the solution decay exponentially
in time, as long
as the 0-th wave number remains large enough: that the 0-th Fourier coefficient remains large enough (and does not blow-up)
is the purpose of Lemma \ref{exponentialconvergencetoconstant}.

We want to point out that more than going carefully over the calculations needed in our proofs, we want to 
convey to the reader the spirit of our arguments and of their applicability. More refined calculations can
be found on \cite{cortissoz13, cortissoz17}, where the ideas to be found here have been applied to
other flows.
 
From now on, all the hypothesis of Theorem \ref{maintheorem} on the structure of (\ref{eq1:curvatureflow}) are assumed. 
\begin{lema}
Given a smooth function $\psi$ (which from now on we assume to be the curvature function of a simple smooth closed convex curve),
there is a $\delta>0$ (which may depend on $\left\|\psi\right\|_{\infty}$) such that if $\psi$ satisfies 
\[
\left\|\psi\right\|_{2p+1}\leq \delta\hat{\psi}\left(0\right),
\]
and if on $\left[0,\tau\right]$, the solution to (\ref{eq1:curvatureflow}) with initial condition $\psi$ satisfies
$\left(1-4\delta\right)\hat{\psi}\left(0\right)\leq\hat{k}\left(0,t\right)\leq \left(1+4\delta\right)\hat{\psi}\left(0\right)$, 
and $\left|\hat{k}\left(\pm 1,t\right)\right| \leq 2\delta\hat{k}\left(0,t\right)$. Then, $n^{2p}\left|\hat{k}\left(n,t\right)\right|<\delta \hat{k}\left(0,t\right)$
on the same time interval.
\end{lema}
\begin{proof}
Assume that 
$n^{2p}\left|\hat{k}\left(n,\tau'\right)\right|=\delta \hat{k}\left(0,\tau'\right)$ for a given $n$ occurs for a first time at time $\tau'>0$
with $\tau'<\tau$.
Let us consider the real and imaginary parts of the evolution equation
\[
\begin{aligned}
\frac{d}{dt}\hat{k}(n,t) &= P\hat{k}\left(n,t\right)\\
& \quad- \sum_{q \in \mathcal{Z}-\{n\}} q^{2p}\hat{k}(q,t)\hat{k}^{*M}(n - q,t) + \sum_{A_n \cap \mathcal{Z}^{2p-1}} a_J i^{J \cdot R} Q^R \hat{k}(Q,t).
\end{aligned}
 \]
 where
\[
\begin{aligned}
P_n&=-n^{2p}\hat{k}^{M}(0,t) + 
\sum_{l=1}^{p-1}\sum_{j_l}\left(-1\right)^ln^{2l}a_{j_l}\hat{k}^{j_l}\left(0,t\right).\\
\end{aligned}
\]
Do keep in mind that $P_n$ is real valued.
Now we examine the terms different from $P_n\hat{k}(n,t)$ in the last equation. We begin with
\[
\left|-\sum_{q \in \mathcal{Z}-\{n\}} q^{2p}\hat{k}(q,t)\hat{k}^{*M}(n - q,t) \right|. 
\]
The condition $q\neq n$ in the sum implies that in the convolution term 
$\hat{k}^{*M}(n - q,t)$
there is at least one term $\hat{k}\left(q_j,t\right)$
with $q_j\neq 0$
(if all the terms in this convolution term are $\hat{k}\left(0,t\right)$, then necessarily $q=n$). 
Hence, inside the sum there are two terms of order $\delta$. The fact that 
$\left|q\right|^{2p+1}\left|\hat{k}\left(q,t\right)\right|<\delta\hat{k}\left(0,t\right)$ 
and $\hat{k}\left(0,t\right)\leq \left(1+4\delta\right)\hat{\psi}\left(0\right)$ for $t\in\left[0,\tau'\right]$ imply that the sum 
can be bounded independently of $\mathcal{Z}$
(recall the 'important remark'). Hence this term is of order $O\left(\delta^2/n\right)$ and that the implicit constants
in this bound may depend on powers of $\hat{\psi}\left(0\right)$.

Now we look at the term
\[
\biggl|\sum a_J i^{J \cdot R} Q^R \hat{k}(Q,t)\biggr|.
\]
In this case, by the definition of the set $A_n \cap \mathcal{Z}^{2p-1}$, in the convolution $\hat{k}(Q,t)$ there 
are at least two terms $\hat{k}\left(q_1,t\right)$ and $\hat{k}\left(q_2,t\right)$ with
$q_1,q_2\neq 0$. Also,
the fact that in the expression $Q^R$ the powers of the individual terms are less than or equal to $2p-1$ 
together with that on the interval $\left[0,\tau'\right]$, $\left|q\right|^{2p+1}\left|\hat{k}\left(q,t\right)\right|\leq 
2\delta \hat{k}\left(0,t\right)$ and $\hat{k}\left(0,t\right)\leq \left(1+4\delta\right)\hat{\psi}\left(0\right)$,
imply that the sums can be bounded independently of $\mathcal{Z}$, and hence, by our previous observation,
this term is $O\left(\delta^2/n\right)$ and the implicit constant may depend on powers of $\hat{\psi}\left(0\right)$.

By the hypothesis of Theorem \ref{maintheorem} (more precisely
that referring to $P_{n,4\delta}$), which are in place here, it is not difficult to see that there 
is a $c<0$ such that for all $n$, $P_n\leq c<0$, and for $n$ 
large $P_{n,4\delta}\sim n^{2p}$
and that $\hat{k}\left(n,t\right)$ is of order $\delta/n^{2p+1}$,
at time $\tau'$.
Hence, by taking $\delta > 0$ small enough, the real and imaginary parts of $\dfrac{d}{dt}\hat{k}(n,\tau')$ are dominated by the 
term $P_n\hat{k}(n,\tau')$, which 
is of order $\delta/n$, and this implies that the absolute values of the real and imaginary parts of $\hat{k}(n,t)$ are non increasing
at $\tau'$. This implies the lemma.

\end{proof}
Using the equation $U_{\pm}\left(k\left(\cdot,t\right)\right)=0$, we can show how to control 
the $\pm 1$ wavenumbers. That is the content of the next lemma.
\begin{lema}
Given a function $\psi$,
there is a $\delta>0$ (which may depend on $\left\|\psi\right\|_{\infty}$) such that if 
\[
\left\|\psi\right\|_{2p+1}\leq \delta\hat{\psi}\left(0\right),
\]
and on $\left[0 ,\tau\right]$, 
a solution to (\ref{eq1:curvatureflow}) with initial data $\psi$ satisfies
$\left(1-4\delta\right)\hat{\psi}\left(0\right)\leq\hat{k}\left(0,t\right)\leq \left(1+4\delta\right)\hat{\psi}\left(0\right)$, 
and $n^{2p}\left|\hat{k}\left(n,t\right)\right|<\delta \hat{k}\left(0,t\right)$. Then, 
$\left|\hat{k}\left(\pm 1,t\right)\right|<\delta \hat{k}\left(0,t\right)$ on the same time interval.
\end{lema}
\begin{proof}
The proof is almost identical to the proof of Lemma 3 given in \cite{cortissoz17}, but the idea behind the
proof is quite simple:

Indeed, as we said before, our starting point are the identities $U_{\pm}\left(k\right)=0$. Let us work with
$U_+\left(k\right)=0$. Then, by hypothesis, we can write
\[
k=\hat{k}\left(0,t\right)
\left(1+A\left(t\right)e^{-i\theta}+A^*\left(t\right)e^{i\theta}+O\left(\delta\right)\right),
\]
and here $A^*$ denotes the conjugate of $A$. We may assume, again by the hypotheses,
that $A$ is small in the interval $\left[0,\tau\right]$
(or in a given subinterval), so that
using a Taylor expansion, we can compute that
\[
\int_0^{2\pi}\frac{e^{i\theta}}{k\left(\theta,t\right)}\,d\theta=\frac{A\left(t\right)}{\hat{k}\left(0,t\right)}
+O\left(\delta\right).
\]
But the integral is 0 by hypothesis, and from this we can conclude that
\[
A\left(t\right)=O\left(\delta \hat{k}\left(0,t\right)\right).
\]

\end{proof}
Then we have the following version of the Trapping Lemma
on $\left[0,\tau\right]$:
\begin{lema}[Trapping Lemma]
\label{trappinglemma}
Given a function $\psi$,
there is a $\delta>0$ (which may depend on $\left\|\psi\right\|_{\infty}$),
such that if the initial datum $\psi$ satisfies the following inequality
\[
\left\|\psi\right\|_{2p+1}\leq \delta\hat{\psi}(0),
\]
and such that if for all $t\in \left[0,\tau\right]$
$$\left(1-4\delta\right)\hat{\psi}\left(0\right)\leq\hat{k}\left(0,t\right)\leq \left(1+4\delta\right)\hat{\psi}\left(0\right),$$ 
and
$\left|n\right|^{2p+1}\left|\hat{k}\left(n,t\right)\right|\leq \delta \hat{k}\left(0,t\right)$,
then there exists a $\gamma>0$ that depends on $\delta>0$, and a constant $C>0$ such that the solution to (\ref{eq1:curvatureflow}) satisfies
on $\left[0,\tau\right]$:
\[
\left|\hat{k}\left(n,t\right)\right| \leq \frac{C\delta\hat{\psi}(0)e^{-\gamma \left|n\right|t}}{n^{2p+1}}.
\]
\end{lema}
\begin{proof}
Let be $\hat{v}(\xi,t)=\exp \left( \gamma|\xi|t\right)\hat{k}(\xi,t)$, then:
\[
\begin{aligned}
\hat{v}_t(\xi,t) &= \gamma|\xi|\exp\left(\gamma |\xi| t\right)\hat{k}(\xi,t) + \exp \left(\gamma|\xi|t\right)\hat{k}_t(\xi,t) \\
&=\gamma|\xi|\hat{v}(\xi,t) \\ & \quad + \exp \left( \gamma|\xi|t\right)\left[-\sum q^{2p}\hat{k}(q,t)\hat{k}^{*M}(\xi - q,t) + \sum a_J i^{J \cdot R} Q^R \hat{k}(Q,t)\right]\\
&= \gamma|\xi|\hat{v}(\xi,t) - \frac{\exp \left[\gamma|\xi|t\right)}{\exp \left(\gamma(|\xi -q|+|q|)t\right)}\sum q^{2p}\hat{v}(q,t)\hat{v}^{*M}(\xi - q,t)\\
&\quad + \frac{\exp \left(\gamma|\xi|t\right)}{\exp \left(\gamma|Q|t\right)}\sum a_J i^{J \cdot R} Q^R \hat{v}(Q,t),
\end{aligned}
\]
where $|Q|= \sum_{k,l} |q_{k,l}|$.

For the finite dimensional approximation that we are considering, namely
\[
\begin{aligned}
\frac{d}{dt}\hat{k}(n,t) &=-n^{2p}\hat{k}^{M}(0,t)\hat{k}\left(n,t\right) + 
\sum_{l=1}^{p-1}\sum_{j_l}\left(-1\right)^ln^{2l}a_{j_l}\hat{k}^{j_l}\left(0,t\right)\hat{k}\left(n,t\right)\\
& \quad- \sum_{q \in \mathcal{Z}-\{n\}} q^{2p}\hat{k}(q,t)\hat{k}^{*M}(n - q,t) + \sum_{A_n \cap \mathcal{Z}^{2p-1}} a_J i^{J \cdot R} Q^R \hat{k}(Q,t),
\end{aligned}
\]
the above equation can be rewritten as
\[
\begin{aligned}
\partial_t \hat{v} (n,t) &= \gamma |n| \hat{v}(n,t) \\ &\quad + \exp \left( \gamma |n| t\right)\biggl[-n^{2p}\hat{k}^{M}(0,t)\hat{k}\left(n,t\right) + 
\sum_{l=1}^{p-1}\sum_{j_l}\left(-1\right)^ln^{2l}a_{j_l}\hat{k}^{j_l}\left(0,t\right)\hat{k}\left(n,t\right) \\& \\ &\quad  - \sum_{q \in \mathcal{Z}-\{n\}} q^{2p}\hat{k}(q,t)\hat{k}^{*M}(n - q,t) + \sum_{A_n \cap \mathcal{Z}^{2p-1}} a_J i^{J \cdot R} Q^R \hat{k}(Q,t)\biggr]\\
&= \gamma |n| \hat{v}(n,t) -n^{2p}\hat{v}^{M}\left(0,t\right)\hat{v}\left(n,t\right)+
\sum_{l=1}^{p-1}\sum_{j_l}\left(-1\right)^ln^{2l}a_{j_l}\hat{v}^{j_l}\left(0,t\right)\hat{v}\left(n,t\right)\\ &\quad - \frac{\exp \left(\gamma|n|t\right)}{\exp \left(\gamma[|n -q|+|q|]t\right)}\sum_{q \in \mathcal{Z}-\{n\}} q^{2p}\hat{v}(q,t)\hat{v}^{*M}(n - q,t) \\ &\quad + \frac{\exp \left(\gamma|n|t\right)}{\exp \left(\gamma|Q|t\right)}\sum_{A_n \cap \mathcal{Z}^{2p-1}} a_J i^{J \cdot R} Q^R \hat{v}(Q,t). 
\end{aligned}
\]

Consider the set
\[
\Omega_{\mathcal{Z}}=\left\{w\in\mathbb{C}^{\mathcal{Z}}:\, 2\delta\hat{\psi}\left(0\right)\geq \left\|w\right\|_{2p+1}\right\}.
\]
We shall argue that when $\hat{v}$ hits the boundary of $\Omega_{\mathcal{Z}}$
at time $t$ then $\partial_t\hat{v}$ points towards the interior of $\Omega_{\mathcal{Z}}$. 
To proceed,
let $n$ be a witness of this, that is to say, let
$n\in\mathcal{Z}$ be such that $\left|n\right|^{2p+1}\left|\hat{v}\left(n,t\right)\right|=2\delta\hat{\psi}\left(0\right)$.
First, we shall show that the term
\begin{equation}
\label{rhside}
-n^{2p}\hat{v}^{M}\left(0,t\right)\hat{v}\left(n,t\right)+\gamma\left|n\right|+
\sum_{l=1}^{p-1}\sum_{j_l}\left(-1\right)^ln^{2l}a_{j_l}\hat{v}^{j_l}\left(0,t\right)\hat{v}\left(n,t\right)
\end{equation}
dominates over the other terms in the expression on the left hand side of the equation given

Notice that by the assumption
$$\left(1-4\delta\right)\hat{\psi}\left(0\right)\leq\hat{k}\left(0,t\right)\leq \left(1+4\delta\right)\hat{\psi}\left(0\right),$$
and the assumption on $P_{n,4\delta}$ in Theorem \ref{maintheorem} (recall that $\hat{v}\left(0,t\right)=\hat{k}\left(0,t\right)$), 
for a well chosen $\gamma$,
\[
-n^{2p}\hat{v}^{M}\left(0,t\right)+\gamma\left|n\right|+
\sum_{l=1}^{p-1}\sum_{j_l}\left(-1\right)^ln^{2l}a_{j_l}\hat{v}^{j_l}\left(0,t\right)\leq c < 0,
\]
and $\hat{v}\left(n,t\right)=O\left(\delta/n^{2p+1}\right)$, so this term is of order $\delta/n$ and its real
and imaginary parts have the opposite sign to the real and imaginary parts of $\hat{v}\left(n,t\right)$: 
This fact,
if (\ref{rhside}) were the only term we had to deal with, would imply that $\partial_t\hat{v}$ points
towards the interior of $\Omega_{\mathcal{Z}}$. However, this is not the case and
we must examine other terms and show that they are smaller than (\ref{rhside}) in absolute value when
$\delta$ is small.

Let us start by the term
\begin{equation*}
\left|\frac{\exp \left(\gamma|n|t\right)}{\exp \left(\gamma[|n -q|+|q|]t\right)}\sum_{q \in \mathcal{Z}-\{n\}} q^{2p}\hat{v}(q,t)\hat{v}^{*M}(n - q,t)
\right|
\end{equation*}

As we argued before, condition $q\neq n$ implies that in the convolution term there is at least one term $\hat{k}\left(q_j,t\right)$
with $q_j\neq 0$
(if all the terms in this convolution term are $\hat{v}\left(0,t\right)$, then necessarily $q=n$). 
Hence, inside the sum there are two terms of order $\delta$. The fact that
$\left|q\right|^{2p+1}\left|\hat{k}\left(q,t\right)\right|<\delta\hat{k}\left(0,t\right)$
and $\hat{k}\left(0,t\right)\leq \left(1+4\delta\right)\hat{\psi}\left(0\right)$ imply that the sums can be 
bounded above independently of $\mathcal{Z}$. This shows that the term being considered is of order $\delta^2/n$.

 On the other hand,
the term 
\[
\left\vert\frac{\exp \left(\gamma|n|t\right)}{\exp \left(\gamma|Q|t\right)}\sum_{A_n \cap \mathcal{Z}^{2p-1}} a_J i^{J \cdot R} Q^R \hat{v}(Q,t)\right\vert
\]
is relatively easy to deal with, because by the definition of $A_n \cap \mathcal{Z}^{2p-1}$ the sum in any of its terms which, as the 
reader knows, are basically products of $\hat{k}\left(q,t\right)$ there are at least two terms for which 
$q\neq 0$, which amounts to the fact that the whole sum is of order $\delta^2/n$.

Then, we can conclude that for $\delta>0$ small enough $\partial_t\hat{v}$ points toward the interior of $\Omega_{\mathcal{Z}}$ 
whenever $\hat{v}$ belongs to its boundary. Hence, $|\hat{v}(n,t)|$ is decreasing and therefore
\[
\left|\hat{k}\left(n,t\right)\right| \leq \frac{ \delta \hat{\psi}(0)e^{-\gamma \left|n\right|t}}{n^{2p+1}}.
\]
The case $n=\pm 1$ is obtained via the identity $U_{\pm}\left(k\right)=0$,
following the same ideas as in Lemma 3 in \cite{cortissoz17}. In fact we have that, without being to explicit about the constants that
\[
\left|\hat{k}\left(\pm 1,t\right)\right|=O\left(\delta e^{-\gamma t}\right).
\]
\end{proof}

Now we show that the average curvature $\hat{k}\left(0,t\right)$ remains controlled.
\begin{lema}
\label{exponentialconvergencetoconstant}
 Given an initial condition $\psi$,
there is a $\delta>0$ (which may depend on $\left\|\psi\right\|_{\infty}$),
such that
 if the solution to (\ref{eq1:curvatureflow}) with initial data $\psi$ satisfies 
on the time interval $\left[0,\tau\right]$ that
$\left(1-4\delta\right)\hat{\psi}\left(0\right)\leq\hat{k}\left(0,t\right)\leq \left(1+4\delta\right)\hat{\psi}\left(0\right)$, $n^{2p}\left|\hat{k}\left(n,t\right)\right|\leq \delta\hat{k}\left(0,t\right)$, and $\left|\hat{k}\left(\pm 1,t\right)\right|<\delta \hat{k}\left(0,t\right)$, then 
$\left(1-3\delta\right)\hat{\psi}\left(0\right)\leq \hat{k}\left(0,t\right)\leq \left(1+3\delta\right)\hat{\psi}\left(0\right)$ on the same time interval.
\end{lema}
\begin{proof}
With the same proof of Lemma 3.5 of \cite{cortissoz13}, it is easy to see that
\[
\sum q_1^{j_1}q_2^{j_2} \cdots q_l^{j_l}e^{-\mu|q_1|}e^{-\mu|q_2|} \cdots e^{-\mu|q_l|}e^{-\mu|n- q_l - \cdots q_1|} \leq C_{\mu,l,n}e^{-\frac{\mu}{4^{j_1 + j_2 + \cdots + j_l}}|n|}.
\]
Using the above equation, and because
\[
\begin{aligned}
\frac{d}{dt}\hat{k}(0,t) &= - \sum_{q \in \mathcal{Z}} q^{2p}\hat{k}(q,t)\hat{k}^{*M}(-q,t) + \sum_{A_0 \cap \mathcal{Z}^{2p-1}} a_J i^{J \cdot R} Q^R \hat{k}(Q,t),
\end{aligned}
\]
we get that
\[
\frac{d}{dt}\hat{k}\left(0,t\right)=O\left(\delta^2 \hat{\psi}\left(0\right) e^{-\beta \gamma t}\right),
\]
and integrating this equality for $\delta>0$ small enough, the result follows.
\end{proof}
From the previous lemmas we can conclude the following. 
\begin{prop}
Given an initial condition $\psi$,
there is a $\delta>0$ (which may depend on $\left\|\psi\right\|_{\infty}$),
such that if $\psi$ satisfies the following inequality
\[
\left\|\psi\right\|_{2p+1}\leq \delta\hat{\psi}(0),
\]
then the solution to (\ref{eq1:curvatureflow}) with initial data $\psi$
satisfies the following. There is a $\gamma>0$ and a constant $C>0$ such that
\[
\left|\hat{k}\left(n,t\right)\right|\leq Ce^{-\gamma t},
\]
and $\left(1-4\delta\right)\hat{\psi}\left(0\right)\leq \hat{k}\left(0,t\right)\leq \left(1+4\delta\right)\hat{\psi}\left(0\right)$,
as long as the solution to (\ref{eq1:curvatureflow}) exists.
\end{prop}
 As a consequence of the previous proposition,
 the curvature function
 $k$ 
 remains uniformly bounded as long
 as the solution to (\ref{eq1:curvatureflow}) exists, 
 so the flow exists for all time. Also, it follows that $k$ is converging towards a constant smoothly and exponentially. 
 Indeed, by the Fundamental Theorem of Calculus
\[
\hat{k}(0,s) = \hat{k}(0,t) + \int_{t}^{s}\frac{d}{d\tau}\hat{k}(0,\tau)\ d\tau,
\]
and hence, by Lemma \ref{exponentialconvergencetoconstant} we can take the limit as $s\rightarrow\infty$. Denote
\[
\lim_{s\rightarrow\infty}\hat{k}\left(0,s\right)=\hat{k}\left(0,\infty\right),
\]
and then
\[
\hat{k}(0,\infty) = \hat{k}(0,t) + \int_{t}^{\infty}\frac{d}{d\tau}\hat{k}(0,\tau)\ d\tau.
\]
Therefore
\[
\left|\hat{k}(0,\infty) - \hat{k}(0,t)\right|\leq \int_{t}^{\infty}\left|\frac{d}{d\tau}\hat{k}(0,\tau)\right|\, d\tau\leq Ce^{-\beta\gamma t},
\]
which shows our claim.

In order to improve our estimates, we need to consider the integral form of the equation
\[
\begin{aligned}
\partial_t \hat{k} (n,t) &= \biggl[-n^{2p}\hat{k}^{M}(0,t) + 
\sum_{l=1}^{p-1}\sum_{j_l}\left(-1\right)^ln^{2l}a_{j_l}\hat{k}^{j_l}\left(0,t\right)\biggr]\hat{k}(n,t) \\& \quad- \sum_{q \in \mathcal{Z}-\{n\}} q^{2p}\hat{k}(q,t)\hat{k}^{*M}(n - q,t) + \sum_{A_n \cap \mathcal{Z}^{2p-1}} a_J i^{J \cdot R} Q^R \hat{k}(Q,t).
\end{aligned}
\]
For that, we can use the integrating factor method to obtain
\[
\begin{aligned}
&\hat{k}(n,t)\\ &= \Biggl[\int_{\tau}^{t} \left[\sum_{A_n \cap \mathcal{Z}^{2p-1}} a_J i^{J \cdot R} Q^R \hat{k}(Q,\sigma) - \sum_{q \in \mathcal{Z}-\{n\}} q^{2p}\hat{k}(q,\sigma)\hat{k}^{*M}(n - q,\sigma)\right] e^{-\int_{\tau}^{\sigma} \Xi(s) ds} d\sigma \\ & \quad \quad + \hat{k}(n,\tau)\Biggr]e^{\int_{\tau}^{t} \Xi(s) \ ds},
\end{aligned}
\]
where $\Xi(s)$ is given by
\[
-n^{2p}\hat{k}^{M}(0,s) + 
\sum_{l=1}^{p-1}\sum_{j_l}\left(-1\right)^ln^{2l}a_{j_l}\hat{k}^{j_l}\left(0,s\right).
\]
Using the bounds that we found in our lemmas yield
\small
\[
\begin{aligned}
&\left|\hat{k}(n,t)\right|\\
&\leq \biggl[\int_{\tau}^{t} \biggl|\sum_{A_n \cap \mathcal{Z}^{2p-1}} a_J i^{J \cdot R} Q^R \hat{k}(Q,\sigma) - \sum_{q \in \mathcal{Z}-\{n\}} q^{2p}\hat{k}(q,\sigma)\hat{k}^{*M}(n - q,\sigma)\biggr|\left|e^{-\int_{\tau}^{\sigma} \Xi(s) ds}\right| d\sigma\\
&\quad \quad \quad  \quad + \left|\hat{k}(n,\tau)\right|\biggr]\left|e^{\int_{\tau}^{t} \Xi(s) ds}\right|\\\
&\leq K e^{-n^{2p}\int_{\tau}^t\hat{k}^{M}(0,\tau)\,d\tau + 
\sum_{l=1}^{p-1}\sum_{j_l}\left(-1\right)^ln^{2l}a_{j_l}\int_{\tau}^t\hat{k}^{j_l}\left(0,\tau\right)\,d\tau}.
\end{aligned}
\]
\normalsize
The bound $K$ on the outermost time integral can be shown to 
exist due to the decay estimates (which are exponential) that we have shown for
the terms $\hat{k}\left(n,t\right)$.

Hence, by our previous considerations, given $\epsilon>0$, there is $\tau$ such that if $t>\tau$ then
 $$\left(1 - \epsilon\right)\hat{k}\left(0,\infty\right) \leq \hat{k}(0,t) \leq\left(1 + \epsilon\right)\hat{k}\left(0,\infty\right).$$ 
Thence, we can bound
\[
\left|\hat{k}(n,t)\right|\leq K e^{P_{n,\epsilon}[t-\tau]},
\]
where $P_{n,\epsilon}$ is given by
\[
\begin{aligned}
 P_{n,\epsilon}&=-n^{2p}[1-\epsilon]^{M}\hat{k}^{M}(0,\infty) + \\
&\quad\sum_{l=1}^{p-1}\sum_{j_l}\left(-1\right)^ln^{2l}a_{j_l}\left(1+
\sgn\left(\left(-1\right)^l a_{j_l}\right)\epsilon\right)^{j_l}\hat{k}^{j_l}\left(0,\infty\right).\end{aligned}
\]
This finishes the proof of Theorem \ref{maintheorem}, as for $\epsilon>0$ small enough, with $n=2$ we have assumed that
$P_{2,\epsilon}\geq P_{n,\epsilon}$.
\section{Applications}
\subsection{Polyharmonic flows} 
If $s$ is the length of arc parameter and $k$ the curvature, the polyharmonic flow is given by
\[
\frac{\partial}{\partial t}\gamma=\left(-1\right)^p\left[\frac{\partial^{2p}k}{\partial s^{2p}}\right]\nu.
\]
The curvature then evolves by the equation
\[
\frac{\partial k}{\partial t}=\left(-1\right)^p\left[\frac{\partial^{2p+2}k}{\partial s^{2p+2}} + k^2\frac{\partial^{2p}k}{\partial s^{2p}}\right].
\]
Now using the fact that 
\[
\frac{\partial}{\partial s}=\frac{\partial \theta}{\partial s}\frac{\partial}{\partial \theta}=
k\frac{\partial}{\partial \theta},
\]
it is easy to show that the evolution equation satisfied by the curvature in the case of the polyharmonic flow  satisfies an equation of the form of (\ref{eq1:curvatureflow}) plus the structure conditions imposed on it. In fact, if all the operators $\frac{\partial}{\partial \theta}$ jump 
over all the $k$'s we obtain a term
\[
\left(-1\right)^pk^{2p+2}\left(\frac{\partial^{2p+2}k}{\partial \theta^{2p+2}}+
\frac{\partial^{2p}k}{\partial \theta^{2p}}\right)
\]
whereas, if any of the operators  $\frac{\partial}{\partial \theta}$ acts on any of the $k$'s, 
we obtain a term where there is at least a product of two derivatives of $k$
(of perhaps different order). So in general, the curvature function in the polyharmonic flow satisfies an equation
\begin{equation}
\label{structurepolyharmonic}
\frac{\partial k}{\partial t}=
\left(-1\right)^pk^{2p+2}\left(\frac{\partial^{2p+2}k}{\partial \theta^{2p+2}}+
\frac{\partial^{2p}k}{\partial \theta^{2p}}\right)+H\left(k,\dots,\frac{\partial^{2p+1}k}{\partial \theta^{2p+1}}\right)
\end{equation}
where any term $z_0^{\alpha_0}\dots z_{2p+1}^{\alpha_{2p+1}}$ of the polynomial $H\left(z_0,\dots,z_{2p+1}\right)$
satisfies $\alpha_1+\alpha_2+\dots+\alpha_{2p+1}\geq 2$. In this case, we have then that
\[
G=\left(-1\right)^p z_0^{2p+2}z_{2p}+ H.
\]

For instance, for $p=1$ the flow takes the form
\small
\[
\begin{aligned}
\frac{\partial k}{\partial t} &= \left(-1\right)\left[\frac{\partial^{4}k}{\partial s^{4}} + k^2\frac{\partial^{2}k}{\partial s^{2}}\right]\\
&=- k^4\frac{\partial^4 k}{\partial \theta^4} - k^4\frac{\partial^2 k}{\partial \theta^2} - k\left[\frac{\partial k}{\partial \theta}\right]^4 - 11k^2\frac{\partial k}{\partial \theta}\frac{\partial^2 k}{\partial \theta^2} -4k^3\left[\frac{\partial^2 k}{\partial \theta^2}\right]^2 -7k^3\frac{\partial k}{\partial \theta}\frac{\partial^3 k}{\partial \theta^3} - k^3\left[\frac{\partial k}{\partial \theta}\right]^2.
\end{aligned}
\]
\normalsize
It is then clear that $G$ satisfies the conditions required in the introduction.

Applying Theorem \ref{maintheorem}, 
if we start close to the unit circle, in the sense of the seminorms $\left\|\cdot\right\|_{2p+3}$,
the solution to the  polyharmonic flow exists for all time and converges smoothly and
exponentially fast to a circle (and as shown in \cite{Wheeler} to a unit circle if the
area enclosed by the initial curve is $\pi$).

We can also give rates of convergence. For instance, in the case of $p=1$, 
for any $\epsilon>0$ and large enough times (and assuming that the initial curve encloses an area of $\pi$),
\[
\left|\hat{k}(n,t)\right| \leq K e^{\left[-n^{4}[1-\epsilon]^4+n^{2}[1+\epsilon]^4\right][t-\tau]}.
\]
Also, as it is known that if convergent, when the 
initial curve encloses a region of area $\pi$, the polyharmonic flow satisfies $k\rightarrow 1$, then we have that
for any $\epsilon>0$ and large enough times the following estimate holds
\[
\left\|k\left(\theta,t\right)-1\right\|_{C^l\left(\mathbb{S}^1\right)}\leq C_{l,\epsilon}e^{\left[-16[1-\epsilon]^4+4\left(1+\epsilon\right)^4\right]t},
\]
which improves on known results (see \cite{Wheeler}). The general case of this estimate, as described in
Corollary \ref{applicationpolyharmonic} follows from the structure of the evolution 
equation of the curvature given by (\ref{structurepolyharmonic}).
\\

\subsection{Other flows: A very specific example}
As another example, let us consider a simple plane curve $\gamma$ evolving by the following flow
\begin{equation}
\label{nonstandardflow}
\frac{\partial \gamma}{\partial t} = -k\left[\frac{\partial^2 k}{\partial s^2}\right]N,
\end{equation}
where $s$ is the length of arc parameter, $k$ the curvature and $N$ the exterior normal. 
A calculation shows that the curvature evolves by the following equation
\[
\begin{aligned}
\frac{\partial k}{\partial t} &= - k \left[\frac{\partial^4 k}{\partial s^4}\right] - k^3\left[\frac{\partial^2 k}{\partial s^2}\right] - \left[\frac{\partial^2 k}{\partial s^2}\right]^2 - 2\left[\frac{\partial k}{\partial s}\right]\left[\frac{\partial^3 k}{\partial s^3}\right]\\
&= - k^5\left[\frac{\partial^4 k}{\partial \theta^4}\right] - k^5\left[\frac{\partial^2 k}{\partial \theta^2}\right] - 5k^4\left[\frac{\partial^2 k}{\partial \theta^2}\right]^2 - k^4\left[\frac{\partial k}{\partial \theta}\right]^2 \\ &\quad- 2k^2\left[\frac{\partial k}{\partial \theta}\right]^4 - 21k^3\left[\frac{\partial k}{\partial \theta}\right]^2\left[\frac{\partial^2 k}{\partial \theta^2}\right] - 9k^4\left[\frac{\partial k}{\partial \theta}\right]\left[\frac{\partial^3 k}{\partial \theta^3}\right].
\end{aligned}
\]
First of all notice that for a given $\delta>0$ we have that 
\begin{eqnarray*}
P_{x,4\delta}&=&-x^4\left(1-4\delta\right)^5\left(\hat{\psi}\left(0\right)\right)^5+x^2\left(1+4\delta\right)^5\left(\hat{\psi}\left(0\right)\right)^5\\
&=&\left(-x^4\left(1-4\delta\right)^5+x^2\left(1+4\delta\right)^5\right)\left(\hat{\psi}\left(0\right)\right)^5\leq -0.5\left(\hat{\psi}\left(0\right)\right)^5,
\end{eqnarray*}
whenever $0<\delta<\frac{1}{32}$.

As a consequence, our computations give us for this case that for initial conditions near the circle of radius 1, with $\hat{\psi}\left(0\right)=1$, 
the solution to (\ref{nonstandardflow}) converges towards a curve whose curvature converges exponentially towards, say
$k_{\infty}$, and 
for any $\epsilon>0$,
there is a time $T_{\epsilon}>0$ and a constant $K_{\epsilon}$ such that for $t>T_{\epsilon}$  
the following estimate holds
\[
\left|\hat{k}(n,t)\right| \leq K_{\epsilon} e^{\left[-n^{4}[1-\epsilon]^5 + n^{2}[1+\epsilon]^{5}\right]k_{\infty}^5[t-\tau]},
\]
and a convergence rate of the curvature towards $k_{\infty}$ is  given by (that is,
for any $\epsilon>0$ there is a constant $C_{l,\epsilon}$ such that)
\[
\left\|k\left(\theta,t\right)-k_{\infty}\right\|_{C^l\left(\mathbb{S}^1\right)} \leq 
C_{l,\epsilon}e^{k_{\infty}^5\left[-16[1-\epsilon]^5 + 4[1+\epsilon]^5\right]t},
\]
which means that the rate of convergence is as close to $\exp\left(-12 k_{\infty}^5\right)$ as wished.

\subsection{A more general family of examples I}
In this section we will consider flows of the form (\ref{generalflow}) with
\begin{equation}
\label{specialF}
F\left(k,k_{s},\dots, k^{\left(2p\right)}_s\right)=\sum_{j=1}^{2p}a_j\frac{\partial^{2j} k}{\partial s^{2j}}
+H,
\end{equation}
where $H$ is again a polynomial on the derivatives of $k$ of order less than $2p$ and where
we only allow terms which contain products of two or more derivatives of $k$ order bigger or equal than 1.
To give an example, $H$ may contain terms of the form
$k_sk_{ss}$ (as long as $2p\geq 4$) or $k\left(k_s\right)^2$.

For a constant $W$ such that given $\psi$ there exists a $\delta>0$ for which
whenever
\[
W=\hat{\psi}\left(0\right)\geq \delta \left\|\psi\right\|_{2p+3}
\]
then the conclusion of Theorem \ref{maintheorem} holds, then we say that
{\it the flow (\ref{generalflow}) is stable around $W$ (which corresponds to a circle of
radius $W^{-1}$)}.

A straightforward computation shows that the curvature $k$ satisfies an equation
\[
\frac{\partial k}{\partial t}=\frac{\partial^{2}}{\partial s^2}F+k^2F.
\]
Then, the polynomial that determines the stability of the flow is the following
\begin{eqnarray*}
&\left(-1\right)^{p+1}a_p\hat{\psi}\left(0\right)^{2p+2}n^{2p+2}&\\
&+&\\
&\sum_{j=2}^p\left(-1\right)^j\left(a_j\hat{\psi}\left(0\right)^{2}+a_{j-1}\right)\hat{\psi}\left(0\right)^{2j}n^{2j}&\\
&+&\\
&(-1)a_1\hat{\psi}\left(0\right)n^2,&
\end{eqnarray*}
since the negativity of this polynomial for $n\geq 2$ implies the condition
on $P_{n,4\delta}$, for $\delta>0$, small enough needed in the hypothesis of Theorem \ref{maintheorem}.
Notice then that if $W=\hat{\psi}\left(0\right)$ is large enough, the previous polynomial
is strictly negative for $n\geq 2$. Hence we have from Theorem \ref{maintheorem} the following.
\begin{prop}\label{family1}
Assume that $\mbox{sgn}\left(a_p\right)=\left(-1\right)^p$.
There exists an $M$ such that if $W\geq M$ then the flow (\ref{generalflow}) with
$F$ of the form (\ref{specialF}) is stable around $W$.
\end{prop}
This proposition basically tells us that in the case of a flow with $F$ of the form 
(\ref{specialF}), its stability depends on the sign of the derivative of the
largest order, and that it will hold around circles that are small enough.

\subsection{A more general family of examples II}
Now we consider 
\begin{equation}
\label{specialF2}
F\left(k,k_{s},\dots, k^{\left(2p\right)}_s\right)=\sum_{j=0}^{p-1}a_{p-j}k^{2j}\frac{\partial^{2p-2j}k}{\partial s^{2p-2j}}
+H,
\end{equation}
with $H$ as in the previous section. This type of $F$ includes as an example the case of
\[
V=k_{ssss}+k^2k_{ss}-\frac{1}{2}k\left(k_s\right)^2,
\]
which is associated to the (normal) evolution of a curve, which is related to the
steepest gradient flow of the functional (see \cite{AndrewsWheeler})
\[
E\left[\gamma\right]=\int_{\gamma}\left(k_s\right)^2\,ds.
\]

In the case that $F$ has the form (\ref{specialF2}), the polynomial that determines stability is given by
\begin{eqnarray*}
&\left(-1\right)^{p+1}a_p\hat{\psi}\left(0\right)^{2p+2}n^{2p+2}&\\
&+&\\
&\sum_{j=0}^{p-2}\left(-1\right)^{p-j}\left(a_{p-j}+a_{p-j-1}\right)\hat{\psi}\left(0\right)^{2p+2}n^{2p-2j}&\\
&+&\\
&(-1)a_1\hat{\psi}\left(0\right)^{2p+2}n^{2},&
\end{eqnarray*}
and hence, a conclusion we can draw without much difficulty is the following.
\begin{prop}\label{family2}
Assume that $\mbox{sgn}\left(a_p\right)=\left(-1\right)^p$,and
\[
\frac{3}{4}\left|a_p\right|\geq \sum_{j=1}^{p-1}\frac{5}{4^{j+1}}\left|a_{p-j}\right|,
\]
 then the flow (\ref{generalflow}) with
$F$ of the form (\ref{specialF2}) is stable around any $\hat{\psi}\left(0\right)=W>0$.
\end{prop}
Finally, notice that $V$ above satisfies the hypothesis of the proposition.

\subsection{Regularity}

Here we discuss the sharpness of requirement (\ref{initialconditionregularity}) in the statement
of Theorem \ref{maintheorem}.
We start 
with the case of polyharmonic flows, and
we shall center our discussion in the case $p=1$.

We will show that condition (\ref{initialconditionregularity}) which amounts to
\[
\hat{\psi}\left(0\right)\geq \delta\left\|\psi\right\|_{5},
\]
can be replaced by a milder
\[
\hat{\psi}\left(0\right)\geq \delta\left\|\psi\right\|_{\frac{5}{2}+\eta}
\]
where $\eta>0$ is small.

As the reader has seen so far, in the evolution of a Fourier wave number 
there are good and bad terms. By a good term we mean a term whose existence, if it were alone, affects 
the evolution of a Fourier wave number 
$\hat{k}\left(n,t\right)$ by making its norm decrease. The best of these terms
is given by
\begin{equation}
\label{goodterm}
-\hat{k}\left(0,t\right)^4\hat{k}\left(n,t\right),
\end{equation}
which comes from taking the Fourier transform of the term
\[
-k^4\frac{\partial^4 k}{\partial \theta^4}.
\]
Regarding bad terms, those that could offset the evolution of a Fourier wave number 
towards values of smaller norm, 
there is what we might call a very bad term (it has the largest possible order), which in this case is given by a sum
\begin{eqnarray*}
\sum_{q_1,q_2,q_3,q_4}\hat{k}\left(q_1,t\right)\hat{k}\left(q_1,t\right)\hat{k}\left(q_2,t\right)\hat{k}\left(q_3,t\right)
\hat{k}\left(q_4,t\right)&\\
\qquad\left(n-q_1-q_2-q_3-q_4\right)^{4}
\hat{k}\left(n-q_1-q_2-q_3-q_4,t\right),
\end{eqnarray*}
but in this sum, at least there are two $q$'s which are nonzero. If we assume that
\[
\hat{k}\left(q,t\right)\sim \frac{\delta}{\left|q\right|^{\frac{5}{2}+\eta}},
\]
then one can show that the sum above is of order $\delta^2 n^{\frac{3}{2}-\eta}$. But then the
term (\ref{goodterm}) is of order $\delta n^{\frac{3}{2}-\eta}$, so it will dominate the behavior of the
evolution of $\hat{k}\left(n,t\right)$, and then there will be stability for the flow.

As yet another example, working with this heuristics and which can be made rigourous without much difficulty, it
can be shown that the flow studied in the recent paper \cite{AndrewsWheeler}, for which 
\[
F=k_{ssss}+k^2k_{ss}-\frac{1}{2}k\left(k_s\right)^2,
\]
it is enough 
to ask for a condition
\[
\hat{\psi}\left(0\right)\geq \delta\left\|\psi\right\|_{\frac{7}{2}+\eta},
\]
instead of the stronger
\[
\hat{\psi}\left(0\right)\geq \delta\left\|\psi\right\|_{7},
\]
to obtain stability.

\end{document}